%% file: root.tex
\documentclass[letterpaper, 10 pt, conference]{ieeeconf}  

\IEEEoverridecommandlockouts                              
\usepackage[pdftex]{graphicx} 
\usepackage{amsmath} 
\usepackage{amssymb}  
\usepackage{threeparttable}
\usepackage{multirow}
\usepackage{subfigure}
\usepackage{color}
\graphicspath{{./Figures/}}
\usepackage{cite}

\usepackage{amsthm}

\newtheorem{theorem}{Theorem}
\newtheorem{lemma}[theorem]{Lemma}
\newtheorem{definition}[theorem]{Definition}

\newcommand{\norm}[1]{\ensuremath{\left\|#1\right\|}} 

\title{\LARGE \bf
Understanding Robust Control Theory Via Stick Balancing
}

\author{Yoke Peng Leong and John C. Doyle
\thanks{Y. P. Leong and J. C. Doyle are with the Control and Dynamical Systems, California Institute of Technology, 
Pasadena, CA 91125, USA {\tt\small ypleong@caltech.edu, doyle@caltech.edu}.}%
}

\begin{document}

\maketitle
\thispagestyle{empty}
\pagestyle{empty}

\input{intro.tex}

\input{model.tex}

\input{polezero.tex}

\input{bode.tex}

\input{implication.tex}

\input{conclusion.tex}

                                  



\bibliographystyle{IEEEtran}
\bibliography{invertedPendulumref,background,relatedworkref}

\end{document}

%% file: intro.tex
\begin{abstract}

Robust control theory studies the effect of noise, disturbances, and other uncertainty on system performance. Despite growing recognition across science and engineering that robustness and efficiency tradeoffs dominate the evolution and design of complex systems, the use of robust control theory remains limited, partly because the mathematics involved is relatively inaccessible to nonexperts, and the important concepts have been inexplicable without a fairly rich mathematics background. This paper aims to begin changing that by presenting the most essential concepts in robust control using human stick balancing, a simple case study popular in both the sensorimotor control literature and extremely familiar to engineers. With minimal and familiar models and mathematics, we can explore the impact of unstable poles and zeros, delays, and noise, which can then be easily verified with simple experiments using a standard extensible pointer. Despite its simplicity, this case study has extremes of robustness and fragility that are initially counter-intuitive but for which simple mathematics and experiments are clear and compelling. The theory used here has been well-known for many decades, and the cart-pendulum example is a standard in undergrad controls courses, yet a careful reconsidering of both leads to striking new insights that we argue are of great pedagogical value.

\end{abstract}

\section{INTRODUCTION}

A large and growing community in science and engineering utilizes control theory to study the evolution and design of complex systems \cite{Reeves_balance_2013,todorov2005stochastic,sontag2005molecular,cabrera_human_2004,Chandra08072011,li2014robust,doyle2011architecture,stelling2004robustness,braun2011risk,kuo_control_1995}. Noise and other uncertainties are ubiquitous in complex systems, and feedback control is particularly effective in correcting a system that deviates from a planned trajectory due to noise and uncertainties. 

Robust control theory as a powerful tool to investigate the effect of noise, disturbances, and other uncertainties on system performance. Despite its many advantages, robust control theory is minimally used because the mathematics involved is relatively inaccessible to nonexperts, and the important concepts have been inexplicable without a fairly rich mathematics background. This paper aims to begin changing that by presenting the most essential concepts in robust control using human stick balancing. This simple case study is popular in the sensorimotor control literature \cite{Reeves_balance_2013,harrison_pole_2014,milton_delayed_2011}. Moreover, we model it as an inverted pendulum on a moving cart that is extremely familiar to engineers and scientists \cite{kuo2005optimal,Boubaker_survey_2012}. 

The mathematics here is all standard for undergrad engineers, but not for many scientists and physicians, so an important next step is simplifying the underlying math, which is addressed in a related paper \cite{YorieCDC2016}. But this paper should be otherwise read as if aimed at an audience not expert in control, and thus will be more pedantic than would typically be appropriate in a CDC paper.

Robust control theory analyses the amplification of noise and uncertainty in the closed-loop system. This analysis is insightful because it breaks down system performance into noise that is harder to modify and other system parameters that can be redesigned if necessary. Furthermore, robustness is an inherent system property that is independent of the controller design. Hence, one can study the fundamental limits of a closed-loop system without a priori assuming a particular controller architecture. This feature is advantageous, especially in biology, because as better biological control models are developed, the analysis will hold true as long as the underlying physical setup of the system (i.e. physiology and anatomy) is the same.

With the minimal and familiar models, we can explore the impact of unstable poles and zeros, delays, and noise, which can then be easily verified with simple experiments using a standard extensible pointer. Poles and zeros are fundamental quantities in control theory that have observable implications in human stick balancing. Most interestingly, the zero of a system, an often abstract concept, manifests itself as a function of the fixation point of a person during stick balancing. In addition, emergence of oscillations commonly observed in many biological systems \cite{Chandra08072011,mitchell2015oscillatory} when changing certain system parameters may seem cryptic, but it is an unavoidable side effect of degrading system robustness caused by the varying parameters. Despite its simplicity, this case study has extremes of robustness and fragility that are initially counter-intuitive but for which simple math and experiments are clear and compelling.

The theory used here has been well-known for many decades and the cart-pendulum example is a standard in undergrad controls courses, yet a careful reconsidering of both leads to striking new insights that we argue are of great pedagogical value. We hope our approach will help clarify some of the confusions surrounding observations in biology, and bridge the gap between the control community and other research communities.


The rest of the paper is organized as follows: Section \ref{sec:model} presents the stick balancing model as an inverted pendulum on a moving cart. Section \ref{sec:poleandzero} discusses the concept of pole and zero for a system. The concept of system robustness is introduced in Section \ref{sec:robustness} including discussions on the sensitivity function, the Bode's integral, and the waterbed effect. Section \ref{sec:discussion} highlights some important issues in system design, and Section \ref{sec:conclusion} summarizes the ideas in this paper. Some proof are provided for completeness, but readers may skip them and still follow the paper's discussions. 

%% file: model.tex
\section{SIMPLIFIED MODEL}\label{sec:model}


Humans stick balancing is modeled as a one dimensional inverted pendulum on a moving cart. As a model of real three dimensional balancing, it is grossly oversimplified, but will prove to be surprisingly useful for illustrating important concepts in robust control theory. In fact, one can experimentally observe the theoretical predictions using this extremely simple and analytically tractable model.  

\begin{figure}[b!]
	\centering
	\includegraphics[width = \linewidth, clip =true, trim=0.5cm 0cm 0.5cm 1cm]{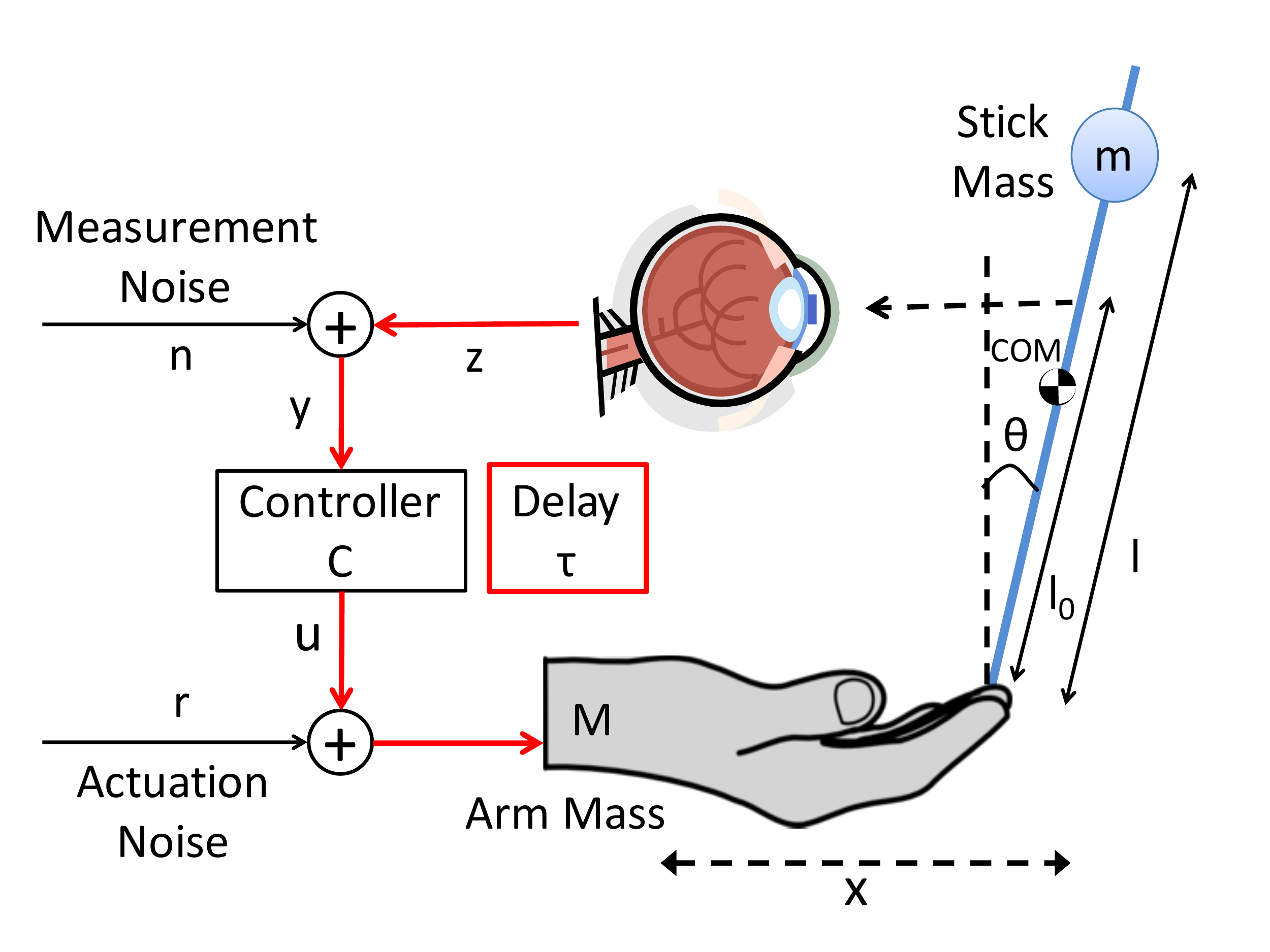}
	\caption{A schematic of balancing an inverted pendulum on one's palm.}
	\label{fig:InvertedPendulum} 
\end{figure}

The dynamics of a one dimensional inverted pendulum on a moving cart (see Fig. \ref{fig:InvertedPendulum}) is given by
\begin{gather}
(M+m)\ddot{x}+m l (\ddot{\theta}\cos\theta-\dot{\theta}^2\sin \theta) =u +r \nonumber\\
m(\ddot{x}\cos \theta + l\ddot{\theta}- g \sin \theta) = 0\nonumber\\
z = x + l_0 \sin \theta \qquad y = z + n
 \label{eq:ndyanmics}
\end{gather}
where $y$ is the position measurement using the eye of $z$, the position of interest, $u$ is the control force, $r$ is the actuation noise, $n$ is the sensor noise, $\theta$ is the pendulum tilt angle from the vertical, and $x$ is the horizontal displacement of the arm. $M$ is the effective mass of the cart (arm or human body), $m$ is the effective mass of the stick, $g$ is the gravitational acceleration, $l_0$ is the fixation point, and $l$ is the effective stick length. Sensorimotor control delay due to signal transmissions and processing is represented as $\tau$ whereby visual processing is the major contributor \cite{lennie_physiological_1981,nijhawan_visual_2008,cavanagh_electromechanical_1979}.

To specify the rigid body motion of a stick and human in term of the point mass system dynamics in \eqref{eq:ndyanmics}, we found the effective masses $m = \frac{3}{4} m'$ and $M = \frac{1}{4} m' +M'$ where $M'$ is the human mass and $m'$ is the stick mass, and the effective stick length $l = \frac{2}{3}l'$ where $l'$ is the actual stick length. Note that center of mass (COM) is below $l$. Henceforth, we refer $m$, $M$, and $l$ as stick mass, cart mass, and stick length respectively. 

\begin{table}[t!]	
	\centering
	\caption{Linearized and Laplace-transformed forms of the cart-pendulum model.}
	\begin{threeparttable}
	\begin{tabular}{cc}
	 \hline \vspace{-0.75em} \\
		Linearized\tnote{*} & Laplace transformed\tnote{\textdagger} \\ \vspace{-0.75em} \\ \hline \vspace{-0.75em} \\
	 $ (M+m)\ddot{x}\pm m l \ddot{\theta} =u+r$ & \multirow{2}{*}{	$\left[\begin{array}{c}
				\hat{x} \\
				\hat{\theta}
				\end{array}\right]= \frac{1}{D(s)} \left[\begin{array}{c}
				l s^2 \mp g \\
				\mp s^2
				\end{array}\right]
				~(\hat{u} + \hat{r}) $} 	 \\
	$m(\pm \ddot{x} + l\ddot{\theta} \mp g\theta) = 0$ &  \\
	$z = x \pm l_0 \theta $ & $\hat{z} = \frac{(l-l_0)s^2\mp g }{D(s)} ~ (\hat{u} + \hat{r})$   \\
    $y = z + n$ & $\hat{y} =  \hat{z} + \hat{n}$ \\  \vspace{-0.75em} \\ \hline	
\end{tabular}
\begin{tablenotes}
       \item[*] The top (bottom) sign in $\pm$ or $\mp$ corresponds to linearization around up (down) equilibrium.
       \item[\textdagger] $D(s) = s^2(M l s^2 \mp (M+m)g) $
     \end{tablenotes}
\end{threeparttable}	
	\label{tab:laplace}
\end{table}

In stick balancing, the controller represents the human brain. The controller's goal is to keep the stick is upright that is the magnitude of $z$ is desired to be small because it corresponds to how far the pendulum has drifted away from the desired equilibrium point. Consequently, we focus on the linearized dynamics of \eqref{eq:ndyanmics} around the equilibrium. Table \ref{tab:laplace} shows the linearization of \eqref{eq:ndyanmics} and the corresponding Laplace-transformed forms. 

%% file: polezero.tex
\section{POLE AND ZERO OF A SYSTEM}\label{sec:poleandzero} 

The pole and zero of a system provide useful insights on the system's behavior. The poles and zeros of the open loop plant in this case study are given in Table \ref{tab:poles_zeros}. Note that the poles and zeros differ depending on whether the stick is pointing upright or downward. 

\begin{table}[b!]
		\centering
		\caption{Poles and zeros of the cart-pendulum model.}
		\begin{tabular}{ccc}
			\multicolumn{3}{c}{Open loop plant: $P(s) =\frac{ (l-l_0)s^2\mp g}{D(s)}$} \vspace{0.5em}\\ \hline \vspace{-0.75em} \\
			Positions & Poles & Zeros \vspace{0.25em}\\ \hline \\
		 &  & $\pm i\sqrt{\frac{g}{l_0-l}}$ if $l_0 > l$ \\
  Upright&   $0$, $\pm \sqrt{\frac{(M+m)g}{M l}}$ &  none if $l_0 = l$\\
	& & $\pm \sqrt{\frac{g}{l-l_0}}$ if $l_0 < l$ \vspace{0.5em} \\ \hline\\
& & $\pm  \sqrt{\frac{g}{l_0-l}}$ if $l_0> l$ \\
Downward &   $0$, $\pm i \sqrt{\frac{(M+m)g}{M l}}$ & none if $l_0 = l$\\
 & & $\pm i \sqrt{\frac{g}{l-l_0}}$ if $l_0< l$ \vspace{0.5em}\\ \hline
	\end{tabular}
	
	\label{tab:poles_zeros}
\end{table}

\subsection{Poles}
Intuitively, stabilizing the stick at the downward position is trivial, but stabilizing the stick at the upright position is comparatively more difficult. The differences in difficulty are captured by the poles. Generally, a pole, represented as a complex number, quantifies the stability of a system. The real part of a pole quantifies system stability, while the imaginary part of a pole represents the resonance frequency. 

The pole for the downward position has a zero real part implying that the system is marginally stable. However, its imaginary part increases with the stick length $l$, and thus, the resonance frequency of the system increases with $l$. 

On the other hand, the pole for the upright position has a positive real part implying that the system is unstable. Hence, the upright position is harder to stabilize than the downward position. The positive pole is generally called a right half plane (RHP) pole. In fact, as the stick length decreases (e.g. shorter stick), the magnitude of the pole increases, and thus, the system is more unstable and harder to control. This last statement may be counter intuitive because a tall object is usually less stable than a short object. Nonetheless, one could easily verify that a shorter stick is harder to balance than a longer stick at home using an extendable stick.

\subsection{Zeros}
A more interesting quantity to consider is the zero. A real positive zero, called a RHP zero, quantifies the effect of measurement. In this case study, zero is related to the fixation point or the location on the stick where a person looks at during stick balancing. When a person look at a certain point on a stick while balancing the stick, the stick may sometimes look stationary although it is moving. The location of this point on the stick defines the zero, and the magnitude of the zero is the critical stick swing frequency when this phenomena occurs. 

During stick balancing, as a person shifts the fixation point from $l$ towards the bottom of the stick, the magnitude of zero decreases from infinity to a small real number. As a result, the critical stick swing frequency decreases and becomes closer to the nominal stick balancing frequency. Therefore, the stick becomes harder to observe and invariably, control. 

Henceforth, we focus on the upright position that resembles human stick balancing and exhibits more interesting behaviors because of the RHP pole and RHP zero. The next section will further quantify the role of RHP pole, RHP zero and delay in system performance using the Bode's integral. 

%% file: bode.tex
\section{BODE'S INTERGRAL - A MEASURE OF ROBUSTNESS} \label{sec:robustness}

\begin{figure}[b!]
	\centering
	\includegraphics[clip =true,  trim=3.3cm 4.5cm 3.3cm 6.5cm, width = \linewidth]{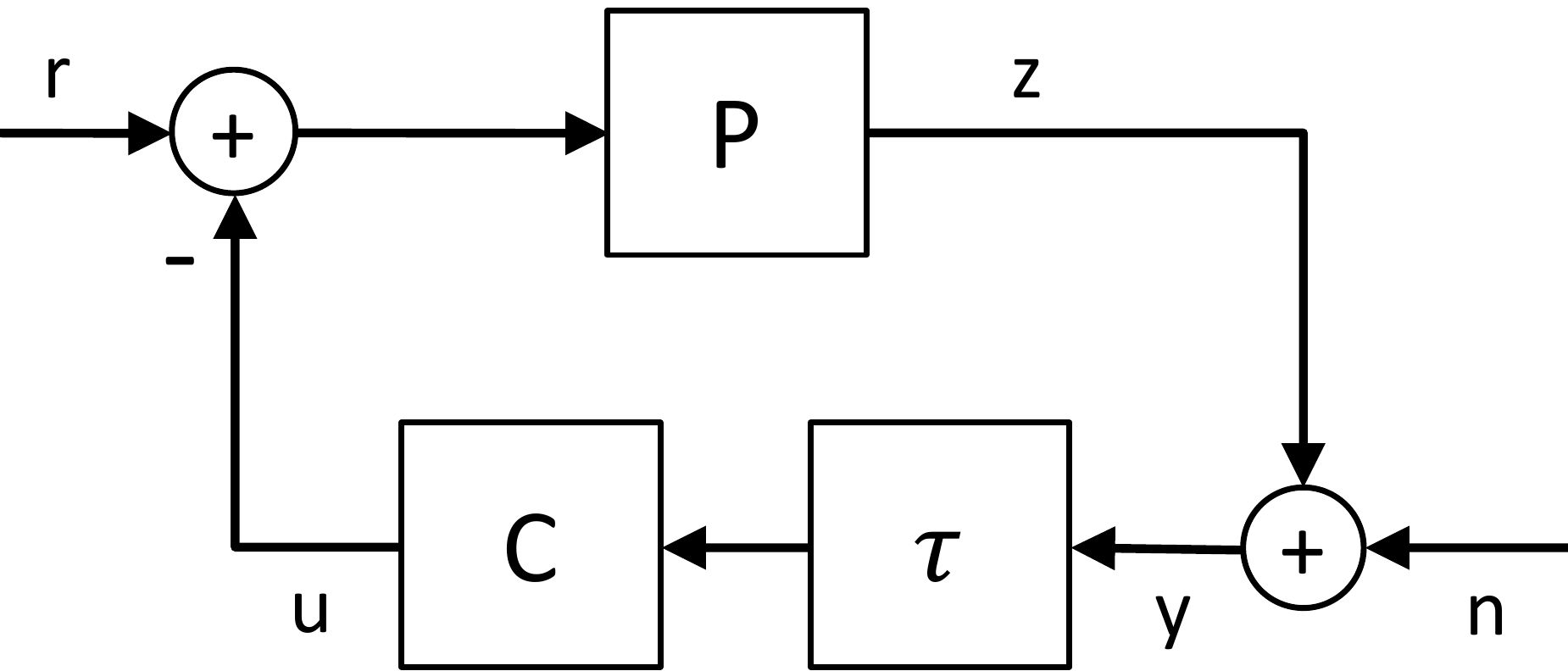}
	\caption{The linear feedback model of the system.}
	\label{fig:InvertedPendulum_block}
\end{figure}

Robust control theory studies the effects of uncertainties in a system, quantifying the system's fundamental limitations. In stick balancing, robustness is characterized by a person's ability to balance a stick. A robust system is equivalent to a setup where a person can balance a stick easily because uncertainties inherited in the human sensorimotor system affect the stick balancing task minimally. A fragile system implies otherwise. More precisely, robust control theory focuses on understanding the amplification of noise, termed robustness, in a closed-loop system and the impacts of RHP poles, RHP zeros, and delay on robustness using the $H_\infty$ norm and the Bode's integral \cite{doyle_dft}. 

\subsection{Sensitivity Functions}

Sensitivity functions are quantities of interest in robust control theory because they are transfer functions from noise to the closed-loop system output. Effectively, they quantify the amplification of noise in a closed-loop system. To obtain the sensitivity functions, the closed-loop output $\hat{z}$ (see Fig. \ref{fig:InvertedPendulum_block}) is derived from equations in Table \ref{tab:laplace}
\begin{equation}
	\hat{z} = \frac{P(s)}{1+P(s)C(s)e^{-\tau s}}\hat{r} - \frac{P(s)C(s)e^{-\tau s} }{1+P(s)C(s)e^{-\tau s}} \hat{n}. \label{eq:clz}
\end{equation}

Define the sensitivity function, $S(s)$, and the complimentary sensitivity function, $T(s)$, \cite{doyle_dft} as
\begin{align}
S(s) = \frac{1}{1+P(s) C(s) e^{-\tau s}}, \quad
T(s) = \frac{P(s) C(s)e^{-\tau s}}{1+P(s) C(s)e^{-\tau s}}. \label{eqn:T} 
\end{align} 
Then, \eqref{eq:clz} can be rewrite as
\begin{equation*}
	\hat{z} =  P(s)S(s)\hat{r} - T(s) \hat{n}\label{eq:clperformance}.
\end{equation*} 
If $T(s)$ and $S(s)$ is small, for any given noises $r$ and $n$, the deviation $z$ will also be small. Note that given a RHP pole $p$ and a RHP zero $q$,
\begin{equation}
	S(q) = 1 \qquad T(p) = 1, \label{eq:st-one}
\end{equation}
an important fact that is used later.

\subsection{$H_\infty$ Norm}\label{sec:hinf norm}
In stick balancing, the deviation $z$ of the stick from the equilibrium point is desired to be small for any given noises $r$ and $n$. Hence, define the performance measure of the system as the infinity norm of the closed-loop transfer functions $T(s)$ and $S(s)$ from noise sources to $z$.
\begin{definition}[$H_\infty$ norm]
The infinity ($H_\infty$)  norm of a transfer function is defined as
\begin{equation*}
	\norm{G}_{\infty} = \sup_\omega |G(j\omega)| = \sup_{Re ~ s >0} |G(s)|
\end{equation*}
where the last equality is given by the Maximum Modulus Theorem \cite{doyle_dft}.
\end{definition}
\noindent By this definition, a system is robust when the infinity norm of $T(s)$ and $S(s)$ are small (i.e. when the magnitude of the transfer function at all frequencies is small). 

In stick balancing, the impacts of the measurement noise and the actuation noise (e.g. neural signal noise during muscle activation) on system robustness are similar, but the feedback delay will only amplify the measurement noise. Hence, we assume that the limiting factor is the sensing noise (see Section \ref{sec:discussion} for more details), and the analysis henceforth will focus on $T(s)$. Nonetheless, similar analysis can be applied to study the effect of actuation noise. The infinity norm of $T(s)$ is given as follows.
\begin{theorem} \label{thm:hinfbound}
The infinity norm of $T(s)$ is bounded by
\begin{gather}
\ln \norm{T(s)}_{\infty} \geq F \triangleq  
\left\{ {\begin{array}{cc}
\tau p&{{l_0} \ge l}\\
{\tau p + \ln \left| {\frac{{p + q}}{{p - q}}} \right|}&{{l_0} < l}
\end{array}} \right. \label{eqn:hinfT} \\
p = \sqrt{\frac{(M+m)g}{M l}} \qquad \quad q  = \sqrt{\frac{g}{l-l_0}} \nonumber
\end{gather}
where $T(s)$ is the transfer function from measurement noise $n$ to output $z$ given in \eqref{eqn:T}, $F$ is the system fragility, $p$ is the RHP pole, $q$ is the RHP zero, and $\tau$ is the total delay in the control system from measurement to actuation.
\end{theorem}

\begin{proof}
	Factor the complimentary sensitivity function as 
	\begin{equation*}
		T(s) = T_{mp}(s) T_{ap}(s) 
	\end{equation*}
	where $mp$ is minimum-phase and $ap$ is all-pass.

	When $l_0 < l$, there is only one RHP zero,
	\begin{equation*}
		T_{ap}(s) = \frac{s-q}{s+q}e^{-\tau s}.
	\end{equation*}
	Then, by \eqref{eq:st-one},
	\begin{equation*}
		T_{mp}(p) = T_{ap}(p)^{-1} = \frac{p+q}{p-q}e^{\tau p}.
	\end{equation*}	
	Hence,
	\begin{align*}
		\norm{T(s)}_{\infty}& = \sup_{Re ~ s >0} |T(s)| =\sup_{Re ~ s >0} |T_{mp}(s)|=\norm{T_{mp}(s)}_{\infty}  \\
		&\ge |T_{mp}(p)|= \left|\frac{p+q}{p-q}e^{\tau p}\right|. 
	\end{align*}
	
	When $l_0\geq l $, there is no RHP zero, 
	\begin{equation*}
		\norm{T(s)}_{\infty} \ge |T_{mp}(p)| = e^{\tau p}.
	\end{equation*}
\end{proof}

The $H_\infty$ norm quantifies the worst case noise amplification, and the lower bound $F$ provides the smallest possible worst case noise amplification, the fundamental limit of this system independent of the controller. In engineering, given this information, system engineers decide on whether this fundamental limit is acceptable. If not, the system needs to be redesigned. In sensorimotor control task, the fundamental limit defines a human's best performance in the task using the noisy nervous system regardless of the decision making process of the brain.
Discussions of the bound with respect to stick balancing will be presented after we introduce the Bode's integral and the waterbed effect because this bound will again present itself. 

\subsection{Bode's Integral}

Similar to the $H_\infty$ norm, the Bode's integral places a hard lower bound on the net achievable robustness of a control system in the presence of measurement noise. A small integral implies a less fragile (i.e. more robust) system is possible, and sensor noise does not necessarily prevent stabilizing the stick. 
\begin{lemma} \label{lem:bodeintegral}
The Bode's integral is given by
\begin{equation}
\ln |G_{mp}(s_0)| = \frac{1}{\pi}\int^\infty_{-\infty} \ln |G(j\omega)| \frac{\sigma_0}{\sigma_0^2+(\omega-\omega_0)^2} d\omega \label{eq:bode0}
\end{equation}
where $s_0 = \sigma_0+\omega_0 j$ with $\sigma_0 > 0$, $G(s)$ is any transfer function, and $G_{mp}(s)$ is a minimum phase transfer function of $G(s)$. 
\end{lemma}
\begin{proof}
	See \cite{doyle_dft}.
\end{proof}

In stick balancing, the Bode's integral of $T(s)$ is bounded as follows.
\begin{theorem} \label{thm:bodeintegral}
The Bode's integral of $T(s)$ is bounded by
\begin{gather}
	\frac{1}{\pi}\int^\infty_{-\infty} \ln |T(j\omega)|~ \frac{p}{p^2+\omega^2} ~d\omega \ge F
\label{eq:bode}
\end{gather}
where $T(s)$ is the transfer function from measurement noise $n$ to output $z$ given in \eqref{eqn:T}, $F$ is the system fragility, and $p$ is the RHP pole. $F$ and $p$ are defined as in \eqref{eqn:hinfT}.
\end{theorem} 

\begin{proof}
Similar to the proof of Theorem \ref{thm:hinfbound},
when $l_0 < l$, there is only one RHP zero,
\begin{equation*}
	T_{mp}(p) = T_{ap}(p)^{-1} = \frac{p+q}{p-q}e^{\tau p}.
\end{equation*}
Substitute into \eqref{eq:bode0},
\begin{align*}
	\frac{1}{\pi}\int^\infty_{-\infty} \ln |T(j\omega)| \frac{p}{p^2+\omega^2} d\omega & =  \ln |T_{mp}(p)| \nonumber\\	 
	&= \tau p + \ln\left|\frac{p+q}{p-q}\right|.
\end{align*}

When $l_0\geq l $, there is no RHP zero, 
\begin{align*}
	\frac{1}{\pi}\int^\infty_{-\infty} \ln |T(j\omega)| \frac{p}{p^2+\omega^2} d\omega 
	= \ln\left( e^{\tau p}\right) = \tau p .
\end{align*}
\end{proof}

\subsection{The Waterbed Effect}

The Bode's integral is a conserved quantity that exhibits the waterbed effect. If $T(s)$ is small for a certain frequency range, then at some other frequency range, $T(s)$ will necessary be large, similar to a waterbed whereby one side will rise when you push at the other side. More precisely,

\begin{theorem}[Waterbed Effect]\label{thm:waterbed}
Suppose that a plant $P(s)$ has a RHP pole $p$. Then, there exist positive constant $c_1$ and $c_2$, depending only on $\omega_1$, $\omega_2$, and $p$, such that
\begin{equation} 
	c_1 \ln M_1 + c_2 \ln M_2 \ge \ln \mid T_{ap}(p) \mid^{-1} \ge 0
\end{equation}
where $[\omega_1,\omega_2]$ defines some frequency range, $T_{ap}(s)$ is the all-pass complimentary sensitivity function for the plant $P(s)$, $M_1 = \max_{\omega_1 \le \omega \le \omega_2} |T(j \omega)|$ and $M_2 = \norm{T(s)}_\infty$.
\end{theorem}

\begin{proof}
Let $p = \sigma_0+j \omega_0$ and recall that $T_{mp}(p)=T_{ap}(p)^{-1}$. Then, from Lemma \ref{lem:bodeintegral},
\begin{align*}
	\ln |T_{ap}(p)^{-1}|& = \frac{1}{\pi}\int^\infty_{-\infty} \ln |T(j\omega)| \frac{\sigma_0}{\sigma_0^2+(\omega-\omega_0)^2} d\omega \nonumber \\
	& \le c_1 \ln M_1 + c_2 \ln M_2. 
\end{align*}
where $c_1$ is the integral of $\frac{1}{\pi}\frac{\sigma_0}{\sigma_0^2+(\omega-\omega_0)^2}$ over the domain $[-\omega_2,-\omega_1] \cup [\omega_1,\omega_2]$ and $c_2$ is the same integral over the complimentary domain.
By Maximum Modulus Theorem, $|T_{ap}(p)^{-1}| \le 1$. Thus, $\ln |T_{ap}(p)^{-1}| \ge 0$.
\end{proof}

Next, we relate Theorem \ref{thm:waterbed} to stick balancing.
\begin{theorem}\label{thm:waterbedip}
	The complimentary sensitivity function $T(s)$ satisfies
\begin{equation} 
	c_1 \ln M_1 + c_2 \ln M_2 \ge F 
\end{equation}	
where $F$ is the system fragility defined as in \eqref{eqn:hinfT}, $M_1 = \max_{\omega_1 \le \omega \le \omega_2} |T(j \omega)|$, $[\omega_1,\omega_2]$ defines some frequency range, $M_2 = \norm{T(s)}_\infty$, and $c_1$ and $c_2$ are constants that depend only on $w_1$, $w_2$, and $p$.
\end{theorem}

\begin{proof}
	Recall that $T_{ap}(p)^{-1} =  \frac{p+q}{p-q}e^{\tau p}$ when there is a RHP zero $q$, and $T_{ap}(p)^{-1} = e^{\tau p}$ when there is no RHP zero. The rest follows from simple algebraic manipulations.
\end{proof}

In human sensorimotor control, if amplification $T(s)$ is small for some frequency range. In other words, $M_1$ is small. Then, by Theorem \ref{thm:waterbedip}, $M_2$ has to be large, implying that $T(s)$ is large at some other frequency. Therefore, if $F$ is large, the system cannot have small $T(s)$ at all frequency. Furthermore, as the bound $F$ increases, the peak of $|T(s)|$ is likely to increase at some frequency. 

\begin{figure}[t!]
	\centering	
	\includegraphics[width = \linewidth,clip =true, trim=1.5cm 6.5cm 2cm 7cm]{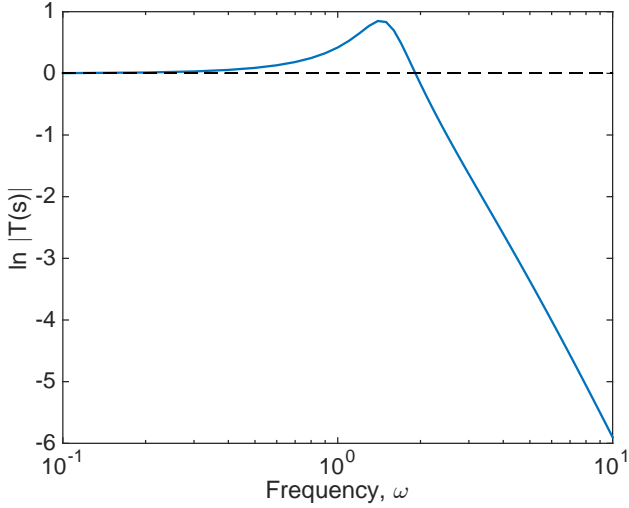}
	\caption{Magnitude of $T(s)$ in log scale as a function of frequency $\omega$ for the case study assuming $C(s) =10$. Note $s = j\omega $.}
	\label{fig:IPBodePlot}
\end{figure}
Fig. \ref{fig:IPBodePlot} shows the magnitude of $T(s)$ in log scale as a function of frequency $\omega$ for the case study assuming $C(s) =10$. The peak of $T(s)$ is within the low frequency range of 1-2 Hz, the stick oscillation frequency observed in experiments. In a more difficult stick balancing setting that causes larger $F$, the stick tends to oscillate at a larger amplitude. Theorem \ref{thm:waterbedip} suggests that the larger oscillation is an unavoidable because the bound $F$ increases.

\subsection{Understanding the Fragility Bound, $F$}
The fragility bound $F$ is a function of delay, RHP pole, and RHP zero which are functions of system parameters including stick length $l$, fixation point $l_0$, stick mass $m$, and cart mass $M$. To further understand the impacts of these system parameters on $F$, $F$ is visualized in Fig. \ref{fig:measure2}, \ref{fig:IPfragility} and \ref{fig:IPHeatPlot}.

\begin{figure}[t!]
	\centering	
	\includegraphics[width=\linewidth,clip =true,  trim=1.5cm 6cm 2cm 7cm]{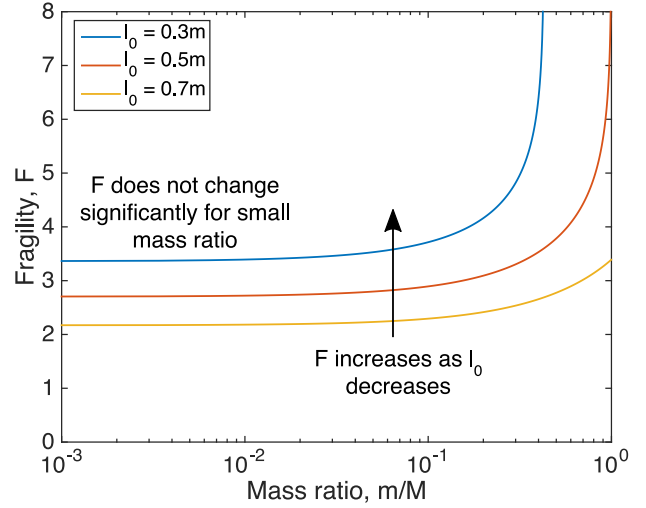}
	\caption{Fragility $F$ is the bound given in \eqref{eqn:hinfT}. Assume $M=3.25$kg, $m = 0.1$kg, and $l=1$m. As the fixation point increases, the fragility curve shifts upwards. Fragility does not vary much when the mass ratio is small.}
	\label{fig:measure2}
\end{figure}

Fig. \ref{fig:measure2} shows that the fragility is minimally affected by stick mass if the stick mass is small. When balancing a stick on one's hand assuming $M = 3.25$kg is the arm mass \cite{clauser1969weight}, the mass ratio is so small that the fragility is barely affected by the stick mass $m$. In practice, $M$ could be either a function of the arm mass or the body mass depending on if the person is allowed to move the body during stick balancing. Either way, $m$ is hardly one-tenth of the $M$. In fact, balancing an extremely heavy stick, for example a $20$kg gym bar ($m = 15$kg), is no more difficult than balancing a light stick. This heavy stick setup has a mass ratio of $m/M = 0.2$ if the person weights $70$kg ($M=75$kg), suggesting that the stick is still easy to balance (see Fig. \ref{fig:measure2}). However, in this situation, the stick becomes strenuous to hold up due to limited muscle strength and fatigue, effects not modeled by $F$.


Fig. \ref{fig:IPfragility} plots fragility $F$ as a function of fixation point $l_0$ or stick length $l$ under four different assumptions. The solid red curve represents $F$ when $l$ is varied, $l_0 = l$, and delay is $0.3$s, a typical value from the literature \cite{lennie_physiological_1981,nijhawan_visual_2008,cavanagh_electromechanical_1979}. The dashed red curve represents $F$ when fixation point $l_0$ is varied, $l=1$m, and delay is $0.3$s. The blue curve represents $F$ when $l$ is varied, $l_0 = l$, and delay is $0.2$s, which is likely to be unrealistically fast, but illustrates the extreme impact of small changes in delay. The yellow curve represents $F$ when $l$ is varied, $l_0 = l$, and delay is $0.5$s. 

\begin{figure}[t!]
	\centering	
	\includegraphics[width = \linewidth,clip =true, trim=1.5cm 6cm 2cm 7cm]{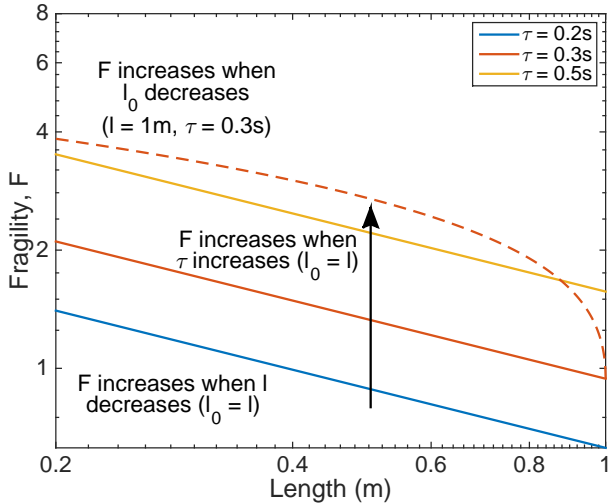}
	\caption{Fragility $F$ is the bound given in \eqref{eqn:hinfT}. Assume $M=3.25$kg, and $m = 0.1$kg. The solid red curve represents $F$ when $l$ is varied, $l_0 = l$ and delay is $0.3$s, the dashed red curve represents $F$ when $l_0$ is varied, $l=1$m, and delay is $0.3$s, the blue curve represents $F$ when $l_0 = l$ and delay is $0.2$s, and the yellow curve represents $F$ when $l_0 = l$ and delay is $0.5$s. Fragility decreases when stick center of mass increases, and fragility increases when delay increases.}
	\label{fig:IPfragility}
\end{figure}

Fig. \ref{fig:IPfragility} emphasizes the influences of stick length, fixation point, and delay. If the stick is short, the fragility increases (see solid lines in Fig. \ref{fig:IPfragility}). In other words, a short stick is harder to balance than a long stick. If the fixation point (while blocking peripheral visions) is moved from $l$ to a lower point along the stick, the stick becomes harder to balance (explained by the dashed red curve in Fig. \ref{fig:IPfragility}). Theoretically, as the fixation point decreases, the system zero moves from infinity towards the pole thus increasing the fragility $F$. 

Interestingly, balancing a stick while focusing at a lower point is much harder than balancing a shorter stick while focusing at the tip (compare both red curves in Fig. \ref{fig:IPfragility}). This effect is more apparent when performing stick balancing while wearing a cap. The cap blocks the peripheral vision from providing information of the stick length $l$ making stick balancing more difficult when fixation point is below $l$. This observation implies that fixation point degrades robustness even more than stick length. 

Lastly, increasing delay, $\tau$, moves the performance curves in Fig. \ref{fig:IPfragility} towards the northeast corner of the plot, and therefore, delay degrades the fundamental limit of this system. This result fits the observation that a trained person can balance the inverted pendulum better than an untrained person because generally, a trained person has a shorter delay in the sensorimotor system \cite{newell1981mechanisms,newell1991motor}. 

Fig. \ref{fig:IPHeatPlot} shows fragility as a function of both $l$ and $l_0$.  It nicely captures the qualitative features of the problem, and even matches qualitatively what can be observed experimentally. It is easy for a person to try to stabilize an extendable pointer in various $(l,l_0)$ points on the figure, and he/she will find that leaving the dark blue region makes stabilization increasingly difficult and ultimately impossible. 

\begin{figure}[t!]
	\centering	
	\includegraphics[width = \linewidth,clip =true, trim=0.5cm 6cm 1.9cm 6.9cm]{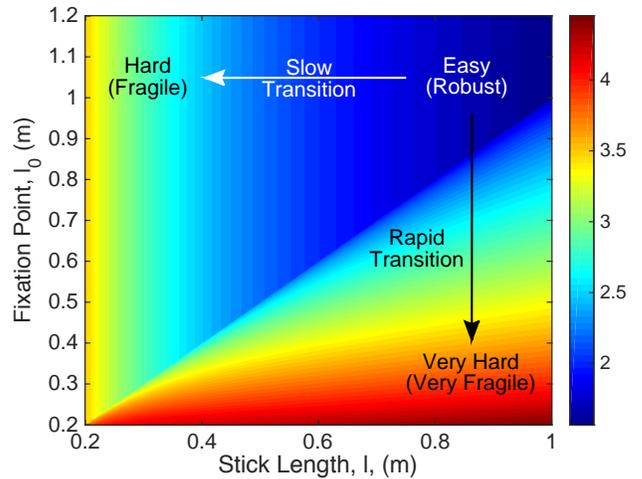}
	\caption{Fragility $F$ is the bound given in \eqref{eqn:hinfT}. Assume $M=3.25$kg, and $m = 0.1$kg. The color corresponds to fragility. Low fixation point affects robustness more than short stick length.}
	\label{fig:IPHeatPlot} 
\end{figure}

%% file: implication.tex
\section{IMPLICATIONS ON SYSTEM DESIGN} \label{sec:discussion}

Apart from providing simple intuitions to important concepts in robust control theory, this case study highlights a few important issues in system design:

\subsection{There exist a fundamental tradeoff between robustness and efficiency.} 
By balancing a long stick and a short stick, one can easily conclude that a short stick is harder to balance than a long stick. When the delay is zero and the RHP zero does not exist, the Bode's integral bound $F$ is constant and zero. But, when a delay is considered, the bound explicitly increases with the RHP pole, which in turn increases as the stick length, $l$, decreases. The theoretical analysis show that given delays, it is impossible to be both robust (balance despite noise) and efficient (short stick). Robots or other organisms would face similar tradeoffs, with the absolute levels determined by controller delay and sensor noise.

\subsection{Delays have an enormous impact on robustness and efficiency tradeoff.}
Delay is fundamental to human sensorimotor control \cite{franklin2011computational} and large scale engineering systems. In human, delays are generally on the order of 100ms, longer for visual system \cite{lennie_physiological_1981,nijhawan_visual_2008,cavanagh_electromechanical_1979}. The simple analysis in this paper suggests that sensorimotor control feedback delay has a detrimental effect on a system as shown by curves in Fig. \ref{fig:IPfragility} moving towards the northeast corner as delay increases.

\paragraph{Vision is the limiting factor instead of actuation because of delay}
In this case study, vision and conscious thought processing are used to stabilize the stick. This task is not what vision is usually used for, and thus, resulting in slower responses. 
Consider this simple visual hand tracking demonstration: (a) Fix your head and shake your hand in front of your eyes; (b) Fix your hand and shake your head looking at your hand. Under same shaking frequency, the hand looks more blurry in (a) than in (b). The reason is because in (a), human tracks the hand using vision that has a large delay, and in (b), human stabilizes the head motion using the fast vestibulo-ocular reflex with minimal delay. In both case, the actuators, the eye muscles, are the same. This simple experiment shows that vision is the main limiting factor instead of actuation. In addition, it shows the role of delay in degrading the performance of a feedback system. 

\paragraph{Distributed architecture results in heterogeneous delays}
A particularly important feature of sensorimotor control and other large scaled systems is that they are \emph{distributed}, an issue that few theories address at all. The sensors, processing, communications, decision making, and actuation are distributed throughout the system in modules that communicate with each other with internal heterogeneous delays. In this case study, only a single lumped source of delay, primarily due to vision, is considered. This idealization is not particularly limiting here, but will be in general. 

\subsection{Sensing location has a significant impact on robustness} 
If the fixation point (while blocking peripheral visions) is moved from the tip of a stick to a lower point along the stick, the stick becomes harder to balance. These observations are consistent with experimental results \cite{Reeves_balance_2013,lee_2012_balancing}. Interestingly, balancing a stick while focusing at a lower point without using peripheral vision is much harder than balancing a shorter stick while focusing at $l$. This observation implies that the fixation point degrades robustness even more than the stick length (i.e. RHP zero is more detrimental than RHP pole). Thus, poorly designed sensing can greatly degrade achievable robustness beyond what is limited by other tradeoffs. 

In fact, beyond stick balancing, researchers recognized gaze location as an important parameters in performing sensorimotor control tasks \cite{marigold2007gaze,vansteenkiste2013visual,baldauf2010attentional,kandil2009driving,land2006eye}. The vision is traditionally passive in that it provides inputs passively to the brain which then processes the information and sends muscle commands to the body. Hence, the goal of a vision system is to obtain images that maximize information relevant to a task. However, we argue that vision has the potential to be more involved in shaping the performance of a system. By choosing the fixation point strategically, one can increase the robustness of a system against uncertainties, and invariably, reduces the effect of noise on the performance of a system. Hence, apart from maximizing information, the vision system could also aim to maximize robustness of a system.
%

\subsection{Oscillation is a result of degrading robustness} 
Many biological systems and ``low level'' physiological processes such as temperature regulation, blood oxygenation, and others, are unstable \cite{Chandra08072011,mitchell2015oscillatory,li2014robust}. These systems are generally controlled automatically and unconsciously. Under specific circumstances, they exhibit oscillation that, we argue, is an inevitable side effect of degrading robustness in an unstable system.

When the fragility is increased for a set of system parameters, the amplitude of oscillations tends to increase as captured mathematically by Theorem \ref{thm:waterbed}. Because the Bode's integral is conserved, the waterbed effect prevents the closed-loop system from having minimal oscillations at all frequencies without fundamentally changing the system design. Note that the waterbed effect in $T(s)$ only applies when the system is unstable. Hence, for an unstable system, oscillation is unavoidable when robustness degrades from changing system parameters under a given circumstance.

%% file: conclusion.tex
\section{CONCLUSIONS} \label{sec:conclusion}
This paper presents fundamental concepts of robust control theory including a discussion on the Bode's integral as a measure of robustness in term of a simple sensorimotor control task, human stick balancing. In this case study, the quantities commonly found in robust control theory and the effects of different system parameters on robustness is easily understandable and reproducible by balancing an extendable stick under a few different settings. Despite its simplicity, important issues in system design can be highlighted including the robustness and efficiency tradeoff, the role of delay, the importance of sensing location, and the effect of fragility on oscillations. 

Lastly, simple modifications while balancing a stick may raise interesting questions that can possibly be explored using robust control theory, such as the effect of peripheral vision, closing one eye, standing on one leg, moving and walking, a darker room, and others. We encourage readers, who are interested in further exploring the limitations of human sensorimotor control, to try these experiments at home.